\documentclass[11pt,letterpaper]{amsart}
\usepackage{amsmath,amsfonts,amssymb,amsthm}

\usepackage{enumerate}
\usepackage{esint}
\usepackage[pagebackref,
colorlinks,citecolor=blue,linkcolor=blue]{hyperref}
\usepackage[dvipsnames]{xcolor}
\usepackage{mathtools}  

\usepackage{mathtools}
\usepackage{setspace}

\usepackage{comment}

\numberwithin{equation}{section}

\theoremstyle{plain}
\newtheorem{theorem}{Theorem}[section]
\newtheorem{corollary}[theorem]{Corollary}
\newtheorem{lemma}[theorem]{Lemma}
\newtheorem{proposition}[theorem]{Proposition}

\theoremstyle{definition}
\newtheorem{definition}[theorem]{Definition}

\newtheorem{example}[theorem]{Example}
\newtheorem{remark}[theorem]{Remark}

\newcommand{\R}{{\mathbb R}}
\newcommand{\N}{{\mathbb N}}

\newcommand{\Om}{\Omega}

\providecommand{\vint}[1]{\mathchoice
	{\mathop{\vrule width 5pt height 3 pt depth -2.5pt
			\kern -9pt \kern 1pt\intop}\nolimits_{\kern -5pt{#1}}}
	{\mathop{\vrule width 5pt height 3 pt depth -2.6pt
			\kern -6pt \intop}\nolimits_{\kern -3pt{#1}}}
	{\mathop{\vrule width 5pt height 3 pt depth -2.6pt
			\kern -6pt \intop}\nolimits_{\kern -3pt{#1}}}
	{\mathop{\vrule width 5pt height 3 pt depth -2.6pt
			\kern -6pt \intop}\nolimits_{\kern -3pt{#1}}}}

\newcommand{\eps}{\varepsilon}
\newcommand{\loc}{\mathrm{loc}}

\newcommand{\ch}{\text{\raise 1.3pt \hbox{$\chi$}\kern-0.2pt}}

\DeclareMathOperator{\Mod}{Mod}

\DeclareMathOperator{\diam}{diam}

\DeclareMathOperator{\lip}{lip}

\begin{document}
\title[Metric Quasiconformality and Sobolev regularity in non-Ahlfors spaces]{Metric quasiconformality and Sobolev regularity in non-Ahlfors regular spaces}
\author{Panu Lahti}
\address{Panu Lahti,  Academy of Mathematics and Systems Science, Chinese Academy of Sciences,
	Beijing 100190, PR China, {\tt panulahti@amss.ac.cn}}
\author{Xiaodan Zhou}
\address{Xiaodan Zhou, Analysis on Metric Spaces Unit, Okinawa Institute of Science and Technology Graduate University, 1919-1, Onna-son, 
Okinawa 904-0495, Japan, {\tt xiaodan.zhou@oist.jp}}


\begin{abstract}
Given a homeomorphism $f\colon X\to Y$ between $Q$-dimensional spaces $X,Y$,
we show that $f$ satisfying the metric definition of quasiconformality outside
suitable exceptional sets implies that $f$ belongs to the Sobolev class $N_{\loc}^{1,p}(X;Y)$, where $1< p\le Q$,
and also implies one direction of the geometric definition of quasiconformality.
Unlike previous results, we only assume a pointwise version of Ahlfors $Q$-regularity,
which particularly enables various weighted spaces to be included in the theory.
Unexpectedly, we can apply this
to obtain results that are new even in the classical Euclidean setting. In particular,
in spaces including the Carnot groups,
we are able to prove the Sobolev regularity $f\in N_{\loc}^{1,Q}(X;Y)$ without the strong
assumption of the infinitesimal distortion $h_f$ belonging to $L^{\infty}(X)$.
\end{abstract}

\date{\today}
\maketitle

\section{Introduction}

Consider two metric spaces $(X,d)$ and $(Y,d_Y)$, and a mapping $f\colon X\to Y$.
For every $x\in X$ and $r>0$, one defines
\[
L_f(x,r):=\sup\{d_Y(f(y),f(x))\colon d(y,x)\le r\}
\]
and
\[
l_f(x,r):=\inf\{d_Y(f(y),f(x))\colon d(y,x)\ge r\},
\]
and then
\[
H_f(x,r):=\frac{L_f(x,r)}{l_f(x,r)};
\]
we interpret this to be $\infty$ if the denominator is zero.
A homeomorphism $f\colon X\to Y$ is (metric) quasiconformal if there is a number
$1\le H<\infty$ such that
\begin{equation}\label{eq:Hf}
H_f(x):=\limsup_{r\to 0} H_f(x,r)\le H
\end{equation}
for all $x\in X$.

In the case where $X$ and $Y$ are Ahlfors $Q$-regular spaces, with $Q>1$,
the homeomorphism $f$ is said to satisfy the \emph{analytic definition} of quasiconformality if
$f\in N^{1,Q}_{\loc}(X;Y)$ and the minimal $Q$-weak upper gradient $g_f$
to the $Q$th power
is bounded pointwise by a constant $C$ times
the Jacobian $J_f$; see Section \ref{sec:Prelis} for definitions.
Moreover,  $f$ is said to satisfy the 
\emph{geometric definition} of quasiconformality if for every family of curves $\Gamma$ in $X$, we have
\[
\frac{1}{C}\Mod_Q(\Gamma)\le \Mod_Q(f(\Gamma))\le C\Mod_Q(\Gamma)
\] 
for some constant $C\ge 1$. 
It has been a problem of wide interest to study the equivalence between these definitions,
and in particular to examine whether the metric definition implies the other definitions.
It has turned out possible to significantly relax the metric definition and still obtain at least
that $f$ is a Sobolev mapping.
Results in this direction have been proven in Euclidean and Carnot-Carath\'eodory spaces by
Gehring \cite{Ge62,Ge},
Marguilis--Mostow \cite{MaMo},
Balogh--Koskela \cite{BaKo},
Kallunki--Koskela \cite{KaKo},
Kallunki--Martio \cite{KaMa}, and
Koskela--Rogovin \cite{KoRo}.
These papers show that one does not need the condition $H_f(x)\le H$ at \emph{every}
point $x$, and that instead of $H_f$ one can consider
\[
h_f(x):=\liminf_{r\to 0} H_f(x,r).
\]
Heinonen--Koskela \cite{HK1,HK2}
and Heinonen--Koskela--Shanmugalingam--Tyson \cite[Theorem 9.8]{HKST} have
studied the equivalence between the definitions in the more general setting
of an Ahlfors $Q$-regular metric space supporting a Poincar\'e inequality.

Balogh--Koskela--Rogovin \cite{BKR} and
Williams \cite{Wi} show that
in order to go from the metric definition of quasiconformality to the fact that $f\in N_{\loc}^{1,1}(X;Y)$,
the assumption of a Poincar\'e inequality is in fact not necessary.
In \cite{LaZh} the authors of the present paper show
that it is possible to largely remove the assumption of Ahlfors regularity as well.
Now we give a similar result when $1<p\le Q$.
In fact we also cover the case $p=1$ but
in this case the assumptions needed in \cite{LaZh} are mostly even weaker than here, so this
not our main focus; the case $p=Q$ is of greatest interest in the theory of quasiconformal mappings.

Our main theorem is the following. 

\begin{theorem}\label{thm:main theorem intro}
 	Let $\Om\subset X$ be open and bounded,
	let $f\colon \Om\to f(\Om)\subset Y$ be a homeomorphism such that
	$f(\Om)$ is open and $\nu(f(\Om))<\infty$,
	and 
	\begin{itemize}
	\item[(1)] Suppose there exists a $\mu$-measurable set $E\subset\Om$ such
	that in
	$\Om\setminus E$ there exist $\mu$-measurable functions $Q(x)>1$ and $R(x)>0$ with
	\[
	\limsup_{r\to 0}\frac{\mu(B(x,r))}{r^{Q(x)}}<R(x)\liminf_{r\to 0}\frac{\nu(B(f(x),r))}{r^{Q(x)}}
	\quad \textrm{for }\mu\textrm{-a.e. }x\in \Om\setminus E.
	\]
	Suppose also that there is a Borel regular outer measure $\widetilde{\mu}\ge \mu$ on $X$ which is doubling
	within a ball $2B_0$ with $\Om\subset B_0$, and
	\[
	\limsup_{r\to 0}\frac{\widetilde{\mu}(B(x,r))}{r^{Q(x)}}<\infty,
	\ \  \liminf_{r\to 0}\frac{\nu(B(f(x),r))}{r^{Q(x)}}>0
	\ \  \textrm{for all }x\in \Om\setminus E.
	\]
		
	\item[(2)]
	Suppose $Q:=\inf_{x\in\Om\setminus E}Q(x)>1$ and let $1\le p\le Q$.
	Assume  that
	\begin{equation}\label{eq:E assumption intro}
	\Mod_p(\{\gamma\subset \Om\colon \mathcal H^1(f(\gamma\cap E))>0\})=0;
	\end{equation}
	in particular, it is enough to assume that $E$ is the union of a countable set and a set
	with $\sigma$-finite codimension $p$ Hausdorff measure.
	
	\item[(3)]
	Finally, assume that $p\le q\le Q$ and that
	\[
	\begin{cases}
	\frac{Q(\cdot)-q}{Q(\cdot)}(R(\cdot)h_f(\cdot)^{Q(\cdot)})^{q/(Q(\cdot)-q)}\in L^1(\Om\setminus E)
	\quad\textrm{if }1\le q<Q\\
	R(\cdot)^{1/Q(\cdot)}h_f(\cdot)\in L^{\infty}(\Om\setminus E)\quad\textrm{if }q=Q.
	\end{cases}
	\]
	\end{itemize}
	Then it follows that $f\in D^{p}(\Om;Y)$.
	
	In the case $p=Q= Q(x)$ for $\mu$-a.e. $x\in \Om\setminus E$,
	for every curve family $\Gamma$ in $\Om$ we also get
	\begin{equation}\label{eq:quasiconformality conclusion}
	\Mod_Q(\Gamma)\le C\Mod_Q(f(\Gamma))
	\end{equation}
	with $C=\Vert R(\cdot)h_f(\cdot)^{Q}\Vert_{L^{\infty}(\Om)}$.
	
	If $X$ is proper and supports a
	$(1,p)$-Poincar\'e inequality, and $\mu$ is doubling, then
	$f\in D^q(\Om;Y)$.
\end{theorem}

Here $D^q(\Om;Y)$ is the Dirichlet space, that is, $f$ is not required to be in $L^q(\Om;Y)$.

While rather technical in its full generality, the theorem will be seen to have
several corollaries that improve on known results;
in Corollary \ref{cor:basic} we make a comparison with Williams \cite[Corollary 1.3]{Wi}.
The main difference with Williams and other previous results
is that instead of local Ahlfors $Q$-regularity we merely
assume the pointwise conditions of (1), as well as the fact that $\mu$
is controlled by a doubling measure $\widetilde{\mu}$.
Note that the ``dimension'' $Q(x)$ as well as the ``density'' $R(x)$
are allowed to vary from
point to point. Mostly we are interested in the case where $Q(x)$ is constant,
but the flexibility provided by the function $R(x)$
makes it easy to include weighted spaces in the theory.

Since we work in spaces without a specific dimension,
we consider a codimension $p$ Hausdorff measure, which however
reduces, up to a constant,  to the $(Q-p)$-dimensional
Hausdorff measure in the Ahlfors $Q$-regular case.
The set $E$ can also contain any countable set; in our generality
even a single point could have infinite codimension $p$ Hausdorff measure.

Similarly to Williams \cite{Wi} and several other previous works,
in \cite{LaZh} we relied on constructing a sequence
of ``almost upper gradients'' $\{g_i\}_{i=1}^{\infty}$.
Since we worked in the case $p=1$, we then had to prove the sequence to be \emph{equi-integrable}
in order to find a weakly converging subsequence.
In the present paper we mostly deal with the case $1<p\le Q$ and so our methods are quite different
from \cite{LaZh}, but partially in the same vein as those of Williams \cite{Wi}.
The key step is to prove boundedness of the sequence $\{g_i\}_{i=1}^{\infty}$
in $L^p$, and then use reflexivity to find a weakly converging subsequence,
and to obtain a $p$-weak upper gradient of $f$ at the limit.

The inequality \eqref{eq:quasiconformality conclusion} is one direction
of the geometric definition of quasiconformality.
There is no hope in general to obtain the opposite estimate
\[
\Mod_Q(f(\Gamma))\le C\Mod_Q(\Gamma),
\]
essentially because all of the assumptions
of Theorem \ref{thm:main theorem intro} still hold if we make $\nu$ bigger,
but this tends to increase $\Mod_Q(f(\Gamma))$.
For a specific counterexample, see Example \ref{ex:plane}.

From Theorem \ref{thm:main theorem intro}, we obtain the following corollary for weighted spaces.

\begin{corollary}\label{cor:weighted}
	Let $1\le p\le q\le Q$.
	Let $(X_0,d,\mu_0)$ and $(Y_0,d_Y,\nu_0)$ be Ahlfors $Q$-regular spaces, with $Q>1$.
	Let $X$ and $Y$ be the same metric spaces
	but equipped with the weighted measures $d\mu=w\,d\mu_0$ and $d\nu=w_Y\,d\nu_0$,
	where $w_Y>0$ is represented by \eqref{eq:wY representative}.
	Let $\Om\subset X$ be open and bounded and let
	$f\colon \Om\to f(\Om)\subset Y$ be a homeomorphism with
	$f(\Om)$ open and $\nu(f(\Om))<\infty$.
	Suppose $w\le \widetilde{w}$ for some weight $\widetilde{w}$ for which
	$d\widetilde{\mu}:=\widetilde{w}\,d\mu_0$ is doubling,
	and suppose there is a set $E\subset \Om$ that is the union of a countable set and a set
	with $\sigma$-finite $\widetilde{\mathcal H}^p$-measure, and
	\[
	\limsup_{r\to 0}\frac{\widetilde{\mu}(B(x,r))}{r^{Q}}<\infty
	\quad \textrm{for all }x\in \Om\setminus E.
	\]
	Finally assume that $h_f<\infty$ in $\Om\setminus E$ and that
	\[
	\begin{cases}
	\left(\frac{w(\cdot)}{w_Y(f(\cdot))}h_f(\cdot)^{Q}\right)^{q/(Q-q)}\in L^1(\Om)\quad\textrm{if }1\le q<Q;\\
	\frac{w(\cdot)}{w_Y(f(\cdot))}h_f(\cdot)^{Q}\in L^{\infty}(\Om)\quad\textrm{if }q=Q.
	\end{cases}
	\]
	Then $f\in D^{p}(\Om;Y)$.
	In the case $p=Q$, for every curve family $\Gamma$ in $\Om$ we also get
	\[
	\Mod_Q(\Gamma)\le C\Mod_Q(f(\Gamma))
	\]
	with $C=\Vert w(\cdot) w_Y(f(\cdot))^{-1} h_f(\cdot)^{Q}\Vert_{L^{\infty}(\Om)}$.
	
	If $X$ is proper and supports a
	$(1,p)$-Poincar\'e inequality, and $\mu$ is doubling, then
	$f\in D^q(\Om;Y)$.
\end{corollary}

We observe that quite general weights $w$ are allowed in the space $X$,
and thus we can include many spaces in the theory where the measure $\mu$ is not Ahlfors regular or even doubling.
Moreover, contrary to all previous works, to the best of our knowledge,
we do not require $h_f$ to be essentially bounded in order to get
$f\in D^{Q}(\Om;Y)$ or $f\in N^{1,Q}(\Om;Y)$.
Instead, $h_f$ can be large in regions where the weight $w$ is small.

In the space $Y$, even more general weights are allowed.
Note that the choice of weight $w_Y$ does not change the classes $D^{p}(\Om;Y)$
or $N^{1,p}(\Om;Y)$, whose definition only depends on $Y$ as a metric space.
Thus we can choose any weight $w_Y$ that is convenient for us.
Already in the classical setting of the unweighted plane, there are simple examples of $N^{1,2}$-mappings
$f\colon \R^2\to\R^2$ for which $h_f$ is not in $L_{\loc}^{\infty}(\R^2)$,
and so the previous results, such as those of \cite{BKR,Wi},
do not tell us that $f\in N^{1,2}_{\loc}(\R^2;\R^2)$.
However, we can detect the Sobolev property of many such mappings
from the next corollary, simply by equipping
$Y$ with a suitable weight $w_Y$;
see Example \ref{ex:plane}.

\begin{corollary}\label{cor:Euclidean}
	Let $1\le q\le n\in\N\setminus\{1\}$.
	Let $w_Y\in L^1_{\loc}(\R^n)$ be represented by \eqref{eq:wY representative}, with
	$w_Y>0$, and $d\nu:=w_Y\,d\mathcal L^n$.
	Let $\Om\subset \R^n$ be open and bounded and let
	$f\colon \Om\to f(\Om)\subset \R^n$ be a homeomorphism with
	$f(\Om)$ open and $\nu(f(\Om))<\infty$.
	Suppose there is a set $E\subset \Om$ that has $\sigma$-finite $\mathcal H^{n-1}$-measure,
	and $h_f<\infty$ in $\Om\setminus E$.
	Finally assume that
	\[
	\begin{cases}
	\left(\frac{h_f(\cdot)^{n}}{w_Y(f(\cdot))}\right)^{q/(n-q)}\in L^1(\Om)\quad\textrm{if }1\le q<n;\\
	\frac{h_f(\cdot)^{n}}{w_Y(f(\cdot))}\in L^{\infty}(\Om)\quad\textrm{if }q=n.
	\end{cases}
	\]
	Then $f\in D^{q}(\Om;\R^n)$.
	In the case $q=n$, for every curve family $\Gamma$ in $\Om$ we have
	\[
	\Mod_n(\Gamma)\le C\Mod_n(f(\Gamma))
	\]
	with $C=\Vert w_Y(f(\cdot))^{-1} h_f(\cdot)^{n}\Vert_{L^{\infty}(\Om)}$.
\end{corollary}

In Carnot groups, we obtain an analogous result.

\begin{corollary}\label{cor:Carnot}
	Let $G$ be a Carnot group of homogeneous dimension $Q>1$.
	Let $w_Y\in L^1_{\loc}(G)$ be represented by \eqref{eq:wY representative}, with
	$w_Y>0$, and $d\nu:=w_Y\,d\mathcal L^Q$.
	Let $\Om\subset \textcolor{blue}{G}$ be open and bounded and let
	$f\colon \Om\to f(\Om)\subset G$ be a homeomorphism with
	$f(\Om)$ open and $\nu(f(\Om))<\infty$.
	Suppose there is a set $E\subset \Om$ that has $\sigma$-finite $\mathcal H^{Q-1}$-measure,
	and $h_f<\infty$ in $\Om\setminus E$.
	Finally assume that
	\[
	\begin{cases}
	\left(\frac{h_f(\cdot)^{Q}}{w_Y(f(\cdot))}\right)^{q/(Q-q)}\in L^1(\Om)\quad\textrm{if }1\le q<Q;\\
	\frac{h_f(\cdot)^{Q}}{w_Y(f(\cdot))}\in L^{\infty}(\Om)\quad\textrm{if }q=Q.
	\end{cases}
	\]
	Then $f\in D^{q}(\Om;G)$.
	In the case $q=Q$, for every curve family $\Gamma$ in $\Om$ we have
	\[
	\Mod_Q(\Gamma)\le C\Mod_Q(f(\Gamma))
	\]
	with $C=\Vert w_Y(f(\cdot))^{-1} h_f(\cdot)^{Q}\Vert_{L^{\infty}(\Om)}$.
\end{corollary}

After giving definitions and notation in Section \ref{sec:Prelis},
we study the exceptional set $E$ in Section \ref{sec:exceptional}.
In Section \ref{sec:preliminary} we study absolute continuity on curves and further preliminary
results, and then in Section \ref{sec:proofs} we prove the results given here in the Introduction.
Finally, in Section \ref{sec:examples} we give examples and applications of our main results.

\section{Definitions and notation}\label{sec:Prelis}

Throughout the paper, we consider two metric measure spaces
$(X,d,\mu)$ and $(Y,d_Y,\nu)$, where $\mu$ and $\nu$ are Borel regular outer measures,
such that the measure of every ball is finite in both spaces.
We understand balls $B(x,r)$, with $x\in X$ and $0<r<\infty$, to be open.
We also assume $X$ to be connected and $Y$ to be separable.
To avoid certain pathologies, we assume that $X$ consists of at least 2 points.

We say that $X$ is metric doubling with constant $M\in\N$ if
every ball $B(x,r)$ can be covered by $M$ balls of radius $r/2$.
This definition works also for subsets $A\subset X$, by considering $(A,d)$ as a metric space.

We will often work also with another Borel regular outer measure $\widetilde{\mu}$ on $X$, which
we assume to satisfy a doubling condition.
We say that $\widetilde{\mu}$ is doubling with constant $C_d\ge 1$ within an open set
$W\subset X$ if
\[
0<\widetilde{\mu}(B(x,2r))\le C_d\widetilde{\mu}(B(x,r))<\infty
\]
for every ball $B(x,r)\subset W$.
For a ball $B=B(x,r)$, we sometimes use the abbreviation $2B=B(x,2r)$;
note that in metric spaces, a ball (as a set) does not necessarily have a unique center point and radius,
but when using this abbreviation we will understand that these have been prescribed.

Given $Q\ge 1$, we say that $\mu$ is locally Ahlfors $Q$-regular 
if for every $z\in X$ there is $R>0$
and a constant $C_A\ge 1$ such that
\[
C_A^{-1}r^Q\le \mu(B(x,r))\le C_A r^Q
\]
for all $x\in B(z,R)$ and $0<r\le R$.

If these conditions hold for every $x\in X$ and $0<r<\diam X$ with $C_d$ (resp. $C_A$)
replaced by a universal constant,
we say that $\mu$ is doubling, or that $\mu$ (or $X$) is Ahlfors $Q$-regular.

Let $\Om\subset X$ always be an open set.
If $f\colon \Om\to f(\Om)\subset Y$ is a homeomorphism with $f(\Om)$ open, 
the pull-back of the measure $\nu$ is the measure on $\Om$ given by
\[
f_\#\nu(D):=\nu(f(D))
\]
for every Borel set $D\subset \Om$.
Note that since $f$ is a homeomorphism, $f_\#\nu$ defines a Borel measure.
The Jacobian of $f$ is then defined by
\begin{equation}\label{eq:Jacobian}
J_f(x):=\lim_{r\to 0}\frac{\nu(f(B(x,r)))}{\mu(B(x,r))}
=\lim_{r\to 0}\frac{f_\#\nu(B(x,r))}{\mu(B(x,r))}
\end{equation}
at every $x\in \Om$ where the limit exists.

The $n$-dimensional Lebesgue measure is denoted by $\mathcal L^n$.

The $s$-dimensional Hausdorff content is denoted by $\mathcal H_R^s$, with $R>0$ and $s\ge 0$,
and the $s$-dimensional Hausdorff measure is denoted by $\mathcal H^s$;
these definitions extend automatically to metric spaces.

For a mapping $f\colon X\to Y$ (resp. $f\colon [a,b]\to Y$), we define
\begin{equation}\label{eq:lip}
\lip_f(x):=\liminf_{r\to 0}\sup_{y\in B(x,r)}\frac{d_Y(f(y),f(x))}{r},\quad x\in X\quad\textrm{(resp. $x\in\R$)}.
\end{equation}

This is easily seen to be a Borel function.

A continuous mapping from a compact interval into $X$ is said to be a rectifiable curve if it has finite length.
A rectifiable curve $\gamma$ always admits an arc-length parametrization,
so that we get a curve $\gamma\colon [0,\ell_{\gamma}]\to X$
(for a proof, see e.g. \cite[Theorem 3.2]{Haj}).
We will only consider curves that are rectifiable and
arc-length parametrized.
If $\gamma\colon [0,\ell_{\gamma}]\to X$ is a curve and
$g\colon X\to [0, \infty]$ is a Borel function, we define
\[
\int_{\gamma} g\,ds:=\int_0^{\ell_{\gamma}}g(\gamma(s))\,ds.
\]

We will always assume that $1\le p<\infty$, though often we will specify a more restricted range for $p$.
The $p$-modulus of a family of curves $\Gamma$ is defined by
\[
\Mod_{p}(\Gamma):=\inf\int_{X}\rho^p\, d\mu,
\]
where the infimum is taken over all nonnegative Borel functions $\rho \colon X\to [0, \infty]$
such that $\int_{\gamma}\rho\,ds\ge 1$ for every curve $\gamma\in\Gamma$.
If a property holds apart from a curve family with zero $p$-modulus, we say that it holds for
$p$-a.e. curve.

Recall that $\Om\subset X$ is always an open set.
Next we give (special cases of) Mazur's lemma and Fuglede's lemma, see e.g. \cite[Theorem 3.12]{Rud}
and \cite[Lemma 2.1]{BB}, respectively.

\begin{lemma}\label{lem:Mazur lemma}
	Let $\{g_i\}_{i\in\N}$ be a sequence
	with $g_i\to g$ weakly in $L^p(\Om)$.
	Then there exist convex combinations $\widehat{g}_i:=\sum_{j=i}^{N_i}a_{i,j}g_j$,
	for some $N_i\in\N$,
	such that $\widehat{g}_i\to g$ in $L^p(\Om)$.
\end{lemma}

\begin{lemma}\label{lem:Fuglede lemma}
Let $\{g_i\}_{i=1}^{\infty}$ be a sequence of functions with $g_i\to g$ in $L^p(\Om)$.
Then for $p$-a.e. curve $\gamma$ in $\Om$, we have
\[
\int_{\gamma}g_i\,ds\to \int_{\gamma}g\,ds\quad\textrm{as }i\to\infty.
\]
\end{lemma}

\begin{definition}\label{def:upper gradient}
Let $f\colon \Om\to Y$. We say that a Borel
function $g\colon \Om\to [0,\infty]$ is an upper gradient of $f$ in $\Om$ if
\begin{equation}\label{eq:upper gradient inequality}
d_Y(f(\gamma(0)),f(\gamma(\ell_{\gamma})))\le \int_{\gamma}g\,ds
\end{equation}
for every curve $\gamma\colon [0,\ell_{\gamma}]\to \Om$.
We use the conventions $\infty-\infty=\infty$ and
$(-\infty)-(-\infty)=-\infty$.
If $g\colon \Om\to [0,\infty]$ is a $\mu$-measurable function
and (\ref{eq:upper gradient inequality}) holds for $p$-a.e. curve in $\Om$,
we say that $g$ is a $p$-weak upper gradient of $f$ in $\Om$.
\end{definition}

We say that $f\in L^p(\Om;Y)$ if $d_Y(f(\cdot),f(x))\in L^p(\Om)$ for some $x\in \Om$.

\begin{definition}
The Newton-Sobolev space $N^{1,p}(\Om;Y)$ consists of those
mappings $f\in L^p(\Om;Y)$ for which there exists an upper gradient $g\in L^p(\Om)$.

The Dirichlet space $D^p(\Om;Y)$ consists of those mappings $f\colon \Om\to Y$
that have an upper gradient $g\in L^p(\Om)$.
\end{definition}

In the classical setting of $X=Y=\R^n$, both spaces equipped with the Lebesgue measure,
and assuming $f$ is  continuous,
we have $f\in N^{1,p}(X;Y)$ if and only if
$f\in W^{1,p}(\R^n;\R^n)$, see e.g. \cite[Theorem A.2]{BB}.\\

We say that $f\in N_{\loc}^{1,p}(\Om;Y)$ if for every $x\in \Om$ there is $r>0$ such that
$f\in N^{1,p}(B(x,r);Y)$; other local spaces of mappings are defined analogously.

It is known that for every $f\in D_{\loc}^{p}(\Om;Y)$ there exists a minimal $p$-weak
upper gradient of $f$ in $\Om$, denoted by $g_{f}$, satisfying $g_{f}\le g$ 
$\mu$-a.e. in $\Om$ for every $p$-weak upper gradient $g\in L_{\loc}^{p}(\Om)$
of $f$ in $\Om$, see \cite[Theorem 6.3.20]{HKSTbook}.

We say that $X$ supports a $(1,p)$-Poincar\'e inequality
if every ball in $X$ has nonzero $\mu$-measure, and
there exist constants $C_P>0$ and $\lambda \ge 1$ such that for every
ball $B(x,r)$, every $u\colon X\to \R$ that is integrable on balls,
and every upper gradient $g$ of $u$,
we have
\begin{equation}\label{eq:Poincare}
\vint{B(x,r)}|u-u_{B(x,r)}|\, d\mu 
\le C_P r\vint{B(x,\lambda r)}g\,d\mu,
\end{equation}
where 
\[
u_{B(x,r)}:=\vint{B(x,r)}u\,d\mu :=\frac 1{\mu(B(x,r))}\int_{B(x,r)}u\,d\mu.
\]
\phantom\\

\section{The exceptional set $E$}\label{sec:exceptional}

In this section we consider some preliminary results and use them to study the exceptional set $E$.
Recall that the standing assumptions on the spaces $X,Y$
are listed in the first paragraph of Section \ref{sec:Prelis}.

In Balogh--Koskela--Rogovin \cite{BKR} and Williams \cite{Wi},
it is assumed that the exceptional set $E$ is $\sigma$-finite with respect
to the $(Q-p)$-dimensional Hausdorff measure. Since we do not assume the space to be
Ahlfors regular, we instead consider a codimension $p$ Hausdorff measure.
Specifically, we consider it with respect to the measure $\widetilde{\mu}$,
which will always be assumed to satisfy a doubling condition.

\begin{definition}\label{def:centered codimension one measure}
	Let $1\le p<\infty$.
	For any set $A\subset X$ and $0<R<\infty$, the restricted Hausdorff content
	of codimension $p$ is defined by
	\[
	\widetilde{\mathcal{H}}^p_{R}(A):=\inf\left\{ \sum_{j}
	\frac{\widetilde{\mu}(B(x_{j},r_{j}))}{r_{j}^p}\colon\,
	\ A\subset\bigcup_{j}B(x_{j},r_{j}),\ r_{j}\le R\right\},
	\]
	where we consider finite and countable coverings. 
	The codimension $p$ Hausdorff measure of $A\subset X$ is then defined by
	\[
	\widetilde{\mathcal{H}}^p(A):=\lim_{R\rightarrow 0}\widetilde{\mathcal{H}}^p_{R}(A).
	\]
\end{definition}

If $\widetilde{\mu}$ is Ahlfors $Q$-regular, then $\widetilde{\mathcal{H}}^p$ is easily seen to be
comparable to $\mathcal H^{Q-p}$ when $1\le p\le Q$.

The following result is often used and a proof can be found e.g. in \cite[Lemma 4.2]{Boj}.

\begin{lemma}\label{lem:dilated balls lemma}
	Let $1<p<\infty$.
	Let $\Om\subset X$ be open and suppose $\widetilde{\mu}$ is doubling with constant $C_d$ within 
	a ball $2B_0$, with $\Om\subset B_0$.
	Then for any finite or countable collection of balls $\{B_j=B(x_j,r_j)\}_{j}$ with
	$6B_j\subset \Om$,
	and for numbers $a_j\ge 0$, we have
	\[
	\int_X \left(\sum_{j}a_j\ch_{6B_j}\right)^p\,d\widetilde{\mu}
	\le C_0\int_X \left(\sum_{j}a_j\ch_{B_j}\right)^p\,d\widetilde{\mu}
	\]
	for a constant $C_0$ that only depends on $C_d$ and $p$.
\end{lemma}


We have the following ``continuity from below'' for the modulus of families of curves;
for a proof see \cite[Lemma 2.3]{Z2}
or \cite[Proposition 5.2.11]{HKSTbook}.

\begin{lemma}\label{lem:mod-continuity}
	Let $1<p<\infty$.
	If $\{\Gamma_j\}_{j=1}^{\infty}$ is a sequence of families of curves such that
	$\Gamma_j\subset \Gamma_{j+1}$ for all $j$, then
	\[
	\Mod_p\Bigg(\bigcup_{j=1}^{\infty}\Gamma_j\Bigg) = \lim_{j\to\infty}\Mod_p(\Gamma_j).
	\]
\end{lemma}

In the case $p=1$, from {\cite[Lemma 5.2]{LaZh}} we have the following result for the exceptional set.
By a slight abuse of notation, we denote the image of a curve $\gamma$ also by the same symbol.

\begin{lemma}\label{lem:image Lebesgue measure zero}
	Let $\Om\subset X$ be open and $f\colon \Om\to Y$ continuous.
	Suppose $\widetilde{\mu}\ge \mu$ is doubling
	within $\Om$, and that $E\subset \Om$ has
	$\sigma$-finite $\widetilde{\mathcal{H}}^1$-measure.
	Then $\Mod_1(\Gamma)=0$ for
	\[
	\Gamma:=\{\gamma\subset \Om\colon \mathcal H^1(f(\gamma\cap E))>0\}.
	\]
\end{lemma}

In the case $p>1$, we get the following result,
which is similar to \cite[Lemma 3.5]{BKR}.

\begin{lemma}\label{lem:image infinite points}
	Let $1<p< \infty$.
	Let $\Om\subset X$ be open and suppose $\widetilde{\mu}\ge \mu$ is doubling
	within 
	a ball $2B_0$, with $\Om\subset B_0$.
	Suppose $E\subset \Om$ is a set with $\widetilde{\mathcal{H}}^p(E)<\infty$.
	Then $\Mod_p(\Gamma)=0$ for
	\[
	\Gamma:=\{\gamma\subset \Om\colon \#(\gamma\cap E)=\infty\}.
	\]
\end{lemma}
Here $\# (\gamma\cap E)$ is the cardinality of $\gamma\cap E$.
\begin{proof}
For $j,k\in\N$, define
\[
\Gamma_{j,k}:=\{\gamma\in\Gamma\colon \exists\, \{x_1,\ldots,x_j\}\in\gamma\cap E
\textrm{ s.t. }d(x_l,x_m)>1/k\textrm{ for all }l\neq m\}.
\]
	Fix $j,k\in\N$.
	There is a (finite or countable) cover of $E$ by balls $\{B_n=B(y_n,r_n)\}_{n}$
	such that $r_n\le (10k)^{-1}$, $B(y_n,6r_n)\subset \Om$, and
	\[
	\sum_{n}\frac{\widetilde{\mu}(B(y_{n},r_{n}))}{r_{n}^p} < \widetilde{\mathcal H}^p(E)+1.
	\]
	By the $5$-covering lemma (see e.g. \cite[p. 60]{HKSTbook}),
	we find an index set $I$
	such that the balls $\{B(y_n,r_n)\}_{n\in I}$ are disjoint, and the balls
	$\{B(y_n,5r_n)\}_{n\in I}$ cover $E$.
	Define the function
	\[
	\rho:=\frac{1}{j}\sum_{n\in I}\frac{\ch_{B(y_{n},6r_{n})}}{r_{n}}.
	\]
	Let $\gamma\in \Gamma_{j,k}$ and consider the points $\{x_1,\ldots,x_j\}$.
	Each of these points is contained in a different ball $B(y_n,5r_n)$, $n\in I$,
	and necessarily travels at least the distance $r_n$ in the ball $B(y_n,6r_n)$.
	Thus
	\[
	\int_{\gamma}\rho\,ds
	=\frac{1}{j}\sum_{n\in I}\int_{\gamma}\frac{\ch_{B(y_{n},6r_{n})}}{r_{n}}\,ds
	\ge \frac{1}{j}\times j \times \frac{r_n}{r_{n}}=1,
	\]
	and so $\rho$ is admissible for $\Gamma_{j,k}$.
	Hence
	\begin{align*}
	\Mod_p(\Gamma_{j,k})
	\le\int_X \rho^p\,d\mu
	&\le \int_X \rho^p\,d\widetilde{\mu}\\
	&\le \frac{1}{j^p}\int_X\left(\sum_{n\in I}\frac{\ch_{B(y_{n},6r_{n})}}{r_{n}}\right)^p\,d\widetilde{\mu}\\
	&\le \frac{C_0}{j^p}\int_X\left(\sum_{n\in I}\frac{\ch_{B(y_{n},r_{n})}}{r_{n}}\right)^p\,d\widetilde{\mu}\quad\textrm{by Lemma }\ref{lem:dilated balls lemma}\\
	&= \frac{C_0}{j^p}\sum_{n\in I}\frac{\widetilde{\mu}(B(y_{n},r_{n}))}{r_{n}^p}\\
	&\le \frac{C_0}{j^p}(\widetilde{\mathcal H}^p(E)+1).
	\end{align*}
	Then by Lemma \ref{lem:mod-continuity}, we get
	\[
	\Mod_p\left(\bigcup_{k=1}^{\infty}\Gamma_{j,k}\right)
	=\lim_{k\to\infty}\Mod_p(\Gamma_{j,k})
	\le \frac{1}{j^p}(\widetilde{\mathcal H}^p(E)+1).
	\]
	Finally note that $\Gamma=\bigcap_{j=1}^{\infty}\bigcup_{k=1}^{\infty}\Gamma_{j,k}$.
	Thus $\Mod_p(\Gamma)=0$.
\end{proof}

\begin{lemma}\label{lem:image uncountable points}
	Let $1<p< \infty$.
	Let $\Om\subset X$ be open and suppose $\widetilde{\mu}\ge \mu$ is doubling within 
	a ball $2B_0$, with $\Om\subset B_0$.
	Suppose $E\subset \Om$ has $\sigma$-finite $\widetilde{\mathcal{H}}^p$-measure.
	Then $\Mod_p(\Gamma)=0$ for
	\[
	\Gamma:=\{\gamma\subset \Om\colon \gamma\cap E\textrm{ is uncountable}\}.
	\]
\end{lemma}
\begin{proof}
We have $E=\bigcup_{k=1}^{\infty}E_k$ with $\widetilde{\mathcal{H}}^p(E_k)<\infty$ for every
$k\in\N$. Define the families
\[
\Gamma_k:=\{\gamma\subset \Om\colon \#(\gamma\cap E_k)=\infty\}.
\]
By Lemma \ref{lem:image infinite points}, we have $\Mod_{p}(\Gamma_k)=0$ for every
$k\in\N$. Clearly $\Gamma\subset \bigcup_{k=1}^{\infty}\Gamma_k$, and so also
$\Mod_p(\Gamma)=0$.
\end{proof}

Now we get the following result for the exceptional set $E$.

\begin{proposition}\label{prop:exceptional set E}
	Let $\Om\subset X$ be open and suppose $\widetilde{\mu}\ge\mu$ is doubling within 
	a ball $2B_0$, with $\Om\subset B_0$.
	Let $1\le p<\infty$.
	Suppose $E=E_1\cup E_2\subset \Om$ such that
	$E_1$ is an at most countable set, and $E_2$ is $\sigma$-finite
	with respect to $\widetilde{\mathcal{H}}^p$.
	Let $f\colon \Om\to Y$ be continuous.
	Then
	\[
	\Mod_p(\{\gamma\subset \Om\colon \mathcal H^1(f(\gamma\cap E))>0\})=0.
	\]
\end{proposition}
\begin{proof}
	For every curve $\gamma\subset \Om$, the set $f(\gamma\cap E_1)$ is at most countable,
	and for $p$-a.e. curve $\gamma\subset \Om$ we have
	$\mathcal H^1(f(\gamma\cap E_2))=0$
	by Lemma \ref{lem:image Lebesgue measure zero} and Lemma \ref{lem:image uncountable points}.
\end{proof}

\begin{remark}\label{rem:on the exceptional set E}
	Let $Q>1$.
	A countable set is always $\sigma$-finite with respect to the $Q-p$-dimensional Hausdorff measure,
	when $1\le p\le Q$. Thus when considering this measure, the set $E_1$ could always be included
	in the set $E_2$ above. However, in our setting it can
	happen that even a single point has infinite $\widetilde{\mathcal H}^p$-measure,
	as can be seen from equipping $\R^n$ with
	the weighted measure $d\mu=d\widetilde{\mu}=|x|^{-\alpha}\,d\mathcal L^n$,
	with $n-p<\alpha<n$.
	
	In an Ahlfors $Q$-regular space, $\widetilde{\mathcal H}^p$ is clearly comparable to
	the $(Q-p)$-dimensional Hausdorff measure, when $1\le p\le Q$.
	Thus in such a space our condition on $E$ reduces to that required in
	\cite{BKR,Wi}.
\end{remark}

\section{Preliminary results}\label{sec:preliminary}

In this section we record and prove further preliminary results, such as covering theorems and
basic results on absolute continuity on curves.

An obvious question concerning the Jacobian \eqref{eq:Jacobian} is the existence of the limit.
For this, we consider the following definitions and facts.
Let $Z$ be a separable metric space.
A closed ball is defined by
\[
\overline{B}(x,r):=\{y\in Z\colon d(y,x)\le r\}.
\]

\begin{definition}
A covering $\mathcal F$ of a set $A\subset Z$, consisting of closed balls $\overline{B}(x,r)$,
is called a fine covering if
\[
\inf\{r>0\colon \overline{B}(x,r)\in\mathcal F\}=0
\]
for every $x\in A$.
\end{definition}

\begin{definition}\label{def:fine covering}
Let $Z$ be equipped with a locally finite Borel regular outer measure $\mu_0$.
We say that $\mu_0$ is a Vitali measure in $Z$ if
for every set $A\subset Z$ and every fine covering $\mathcal F$ of $A$, consisting of closed balls
$\overline{B}$, there is a subcollection $\mathcal G\subset\mathcal F$ consisting
of pairwise disjoint balls such that
\[
\mu_0\left(A\setminus \bigcup_{\overline{B}\in \mathcal G}\overline{B}\right)=0.
\]
\end{definition}

By locally finite, we mean that for every $x\in Z$ there is $r>0$ such that
$\mu_0(B(x,r))<\infty$.
For a proof of the following fact, see e.g. \cite[Theorem 3.4.3]{HKSTbook}.

\begin{proposition}\label{prop:mutilde measure}
Suppose $\widetilde{\mu}$ is a Borel regular, locally finite outer measure on $Z$ and
\[
\limsup_{r\to 0}\frac{\widetilde{\mu}(B(x,2r))}{\widetilde{\mu}(B(x,r))}<\infty
\]
for $\widetilde{\mu}$-a.e. $x\in Z$.
Then $\widetilde{\mu}$ is a Vitali measure in $Z$.
\end{proposition}

In our setting, we obtain the following.

\begin{lemma}
Let $\Om\subset X$ be open and suppose there exists a Borel regular outer measure $\widetilde{\mu}\ge \mu$
on $X$ that is doubling within $\Om$. Then $\mu$ is a Vitali measure in the metric space $(\Om,d)$.
\end{lemma}
Note that in Theorem \ref{thm:main theorem intro} we assume that $\widetilde{\mu}$ is doubling within
$2B_0$ with $\Om\subset B_0$, so then in particular $\widetilde{\mu}$ is doubling within $\Om$.
\begin{proof}
By Proposition \ref{prop:mutilde measure} we know that $\widetilde{\mu}$ is a Vitali measure
in the metric space $(\Om,d)$.
Since $\mu\le\widetilde{\mu}$, from Definition \ref{def:fine covering} it clearly follows that
$\mu$ is then also a Vitali measure in  $(\Om,d)$.
\end{proof}

We have the following Lebesgue--Radon--Nikodym differentiation theorem, see e.g.
\cite[p. 82]{HKSTbook}.

\begin{theorem}\label{thm:Radon-Nikodym}
Suppose $Z$ is equipped with a Vitali measure $\mu_0$, and let $\kappa$
be a Borel regular, locally finite measure on $Z$. Then there exists a decomposition of $\kappa$
into the absolutely continuous and singular parts
\[
d\kappa=d\kappa^a +d\kappa^s=a\,d\mu_0 +d\kappa^s,
\]
where
\begin{equation}\label{eq:kappa a}
a(x):=\lim_{r\to 0}\frac{\kappa(B(x,r))}{\mu_0(B(x,r))}
\quad\textrm{for }\mu_0\textrm{-a.e. }x\in Z.
\end{equation}
\end{theorem}

Returning to our setting,
let $\Om\subset X$ be open and suppose there exists a Borel regular outer measure
$\widetilde{\mu}\ge \mu$ on $X$ which is doubling within $\Om$.
Let $f\colon \Om\to f(\Om)\subset Y$ be a homeomorphism, with $f(\Om)$ open.
We can now decompose
\[
df_{\#}\nu
=d(f_{\#}\nu)^a+d(f_{\#}\nu)^s
=a\,d\mu+d(f_{\#}\nu)^s \quad\textrm{in }\Om.
\]
By \eqref{eq:Jacobian} and \eqref{eq:kappa a} we know that in fact $J_f=a$.
Thus
\[
\nu(f(\Om))=(f_{\#}\nu)(\Om)\ge (f_{\#}\nu)^a(\Om)=\int_{\Om}J_f\,d\mu.
\]

We record:
\begin{equation}\label{eq:integral of Jf}
J_f\textrm{ exists }\mu\textrm{-a.e. in }\Om \quad \textrm{and}\quad \int_{\Om}J_f\,d\mu \le \nu(f(\Om)).
\end{equation}

Later we will use this together with the assumption $\nu(f(\Om))<\infty$.

The following covering results are from \cite{LaZh};
they are similar to Lemma 2.2 and Lemma 2.3 of \cite{BKR}.

\begin{lemma}[{\cite[Lemma 3.1]{LaZh}}]\label{lem:N countable collections}
	Let $A\subset X$ be bounded and metric doubling with constant $M$,
	and let $\mathcal F$ be a collection of balls $\{B(x,r_x)\}_{x\in A}$
	with radius at most $R>0$.
	Then there exist finite or countable subcollections $\mathcal G_1,\ldots,\mathcal G_N$,
	with $N=M^4$ and $\mathcal G_j=\{B_{j,l}=B(x_{j,l},r_{j,l})\}_{l}$,
	such that 
	\begin{enumerate}
		\item $A\subset \bigcup_{j=1}^N\bigcup_{l}B_{j,l}$;
		\item If $j\in\{1,\ldots,N\}$ and $l\neq m$, then $x_{j,l}\in X\setminus B_{j,m}$ and $x_{j,m}\in X\setminus B_{j,l}$;
		\item If $j\in\{1,\ldots,N\}$ and $l\neq m$, then $\tfrac 12  B_{j,l}\cap \tfrac 12  B_{j,m}=\emptyset$.
	\end{enumerate}
\end{lemma}

\begin{lemma}[{\cite[Lemma 3.2]{LaZh}}]\label{lem:disjoint images}
	Consider a collection of balls $\{B_l=B(x_l,r_l)\}_{l}$ contained in an open set
	$\Om\subset X$,
	such that for all $l\neq m$ we have $x_l\notin B_m$
	and $x_m\notin B_l$. Also consider an injection $f\colon \Om\to Y$ such that for all $l$ there is
	$H_l\ge 1$ such that 
	\[
	B\left(f(x_l),\frac{L_f(x_l,r_l)}{H_l}\right)\subset f(B_l).
	\]
	Then
	\[
	B\left(f(x_l),\frac{L_f(x_l,r_l)}{2H_l}\right)\cap B\left(f(x_m),\frac{L_f(x_m,r_m)}{2H_m}\right)=\emptyset
	\quad\textrm{for all }l\neq m.
	\]
\end{lemma}

The following two lemmas concerning absolute continuity on curves are essentially well known,
see e.g. Z\"urcher \cite[Lemma 3.6]{Zu},
but we do not know a source for the precise formulations that we need, so we provide full proofs.
Recall the definition of $\lip_h$ from \eqref{eq:lip}.

\begin{lemma}\label{lem:abs cont initial}
	Let $-\infty<a<b<\infty$.
Let $h\colon [a,b]\to Y$ be a continuous mapping such that $\lip_h(t)<\infty$
for all $x\in A$ for some $\mathcal L^1$-measurable set $A\subset [a,b]$.
Then
\[
\mathcal H^1(h(A))\le \int_{A}\lip_h(t)\,dt.
\]
\end{lemma}
\begin{proof}
We can assume that $A\subset (a,b)$.
Define $\lip_h^{\vee}:=\max\{\lip_h,1\}$.
We let
\[
A_k:=\{t\in A\colon 2^{k-1}\le \lip_h^{\vee} <2^{k}\},\quad k\in\N,
\]
so that $\bigcup_{k=1}^{\infty}A_k=A$.
Fix $\delta>0$.
For each $k\in\N$, choose an open set $U_k$ with $A_k\subset U_k\subset (a,b)$ and
\[
\mathcal L^1(U_k)\le \mathcal L^1(A_k)+\frac{\delta}{2^{2k}}.
\]
For every $t\in A_k$
we note that for all $r>0$ the
set $h(B(t,r))$ is contained in the ball $B(h(t),\sup_{s\in B(t,r)}d_Y(h(s),h(t)))$,
and so we can estimate
\begin{align*}
\liminf_{r\to 0}\frac{\mathcal H_{\delta}^1(h(B(t,r)))}{r}
&\le 2\liminf_{r\to 0}\frac{\sup_{s\in B(t,r)}d_Y(h(s),h(t))}{r}\\
&= 2\lip_h(t)
< 2^{k+1}.
\end{align*}
Thus for each $k\in\N$ and every $t\in A_k$, we find $r_t>0$ such that
\[
B(t,r_t)\subset U_k\quad\textrm{and}\quad \mathcal H_{\delta}^1(h(B(t,r_t)))\le 2^{k+1}r_t.
\]
We can choose
a countable subcollection $\{B_j=B(t_j,r_j)\}_{j=1}^{\infty}$ that covers $A$
and with overlap at most $2$. Thus
\begin{align*}
\mathcal H^1_{\delta}(h(A))
\le \sum_{k=1}^{\infty}\mathcal H^1_{\delta}(h(A_k))
&\le \sum_{k=1}^{\infty}\,\sum_{j\in\N\colon x_j\in A_k}\mathcal H^1_{\delta}(h(B_j))\\
&\le \sum_{k=1}^{\infty}2^{k+1}\sum_{j\in\N\colon x_j\in A_k}r_j\\
&\le \sum_{k=1}^{\infty}2^{k+1}\mathcal L^1(U_k)\\
&\le \sum_{k=1}^{\infty}2^{k+1}\left(\mathcal L^1(A_k)+\frac{\delta}{2^{2k}}\right)\\
&\le 4\int_{A}\lip_h^{\vee}\,dt+2\delta.
\end{align*}
Letting $\delta\to 0$, we get
\[
\mathcal H^1(h(A))\le 4\int_{A}\lip_h^{\vee}\,dt.
\]
This is rather close to the desired result, but we only use this to conclude the following
absolute continuity:
\begin{equation}\label{eq:absolute continuity}
\textrm{if }N\subset A\textrm{ with }\mathcal L^1(N)=0,\textrm{ then }\mathcal H^1(h(N))=0.
\end{equation}
Now fix $\eps>0$.
We can assume that $\lip_h\ch_A\in L^1([a,b])$.
For every $t\in A$, we have
\[
\liminf_{r\to 0}\frac{\mathcal H_{\eps}^1(h(\overline{B(t,r)}))}{r}
\le 2\liminf_{r\to 0}\frac{\sup_{s\in B(t,r)}d_Y(h(s),h(t))}{r}
= 2\lip_h(t).
\]
Thus for every Lebesgue point $t\in A$ of the function of $\lip_h \ch_{A}$, we find an
arbitrarily small radius
$r_t$ such that
\[
\mathcal H_{\eps}^1(h(\overline{B(t,r_t)}))\le 2(\lip_h(t)+\eps)r_t
\le (1+\eps)\int_{B(t,r_t)\cap A}(\lip_h+\eps)\,ds.
\]
By the Vitali covering theorem
(recall Proposition \ref{prop:mutilde measure}),
we find a collection $\{\overline{B_k}=\overline{B(x_k,r_k)}\}_{k=1}^{\infty}$
of disjoint balls covering $A\setminus N$ for some $N\subset A$ with $\mathcal L^1(N)=0$.
Then
\begin{align*}
\mathcal H^1_{\eps}(h(A))
&\le \mathcal H^1_{\eps}(h(A\setminus N))+\mathcal H^1_{\eps}(h(N))\\
&\le \sum_{k=1}^{\infty}\mathcal H^1_{\eps}(h(\overline{B_k}))+0\quad\textrm{by }\eqref{eq:absolute continuity}\\
&\le \sum_{k=1}^{\infty}(1+\eps)\int_{B_k\cap A}(\lip_h+\eps)\,ds\\
&= (1+\eps) \int_{A}(\lip_h+\eps)\,ds.
\end{align*}
Letting $\eps\to 0$, we get the result.
\end{proof}

\begin{lemma}\label{lem:abs cont}
Let $\Om\subset X$ be open,
let $f\colon \Om\to Y$ be continuous, and
let $\gamma\colon [0,\ell_{\gamma}]\to \Om$ be a curve
such that $\lip_f(\gamma(t))<\infty$
for all $t\in A\subset [0,\ell_{\gamma}]$, where $A$ is $\mathcal L^1$-measurable.
Then
\[
\mathcal H^1(f(\gamma(A)))\le \int_{A}\lip_f(\gamma(t))\,dt.
\]
\end{lemma}
\begin{proof}
For the mapping $h:=f\circ \gamma\colon [0,\ell_{\gamma}]\to Y$, by the fact that $\gamma$
is a $1$-Lipschitz mapping, for all $t\in A$ we have
\begin{align*}
\lip_h(t)
&=\liminf_{r\to 0}\sup_{s\in B(t,r)}\frac{d_Y(f\circ \gamma(s),f\circ \gamma(t))}{r}\\
&\le \liminf_{r\to 0}\sup_{y\in B(\gamma(t),r)}\frac{d_Y\big(f(y),f (\gamma(t))\big)}{r}\\
&= \lip_f(\gamma(t)).
\end{align*}
Now the result follows from Lemma \ref{lem:abs cont initial}.
\end{proof}

We will use the following theorem of Williams.
The symbol $g_f$ denotes the minimal $Q$-weak upper gradient of $f$ in $Z$.

\begin{theorem}[{\cite[Theorem 1.1]{Wi12}}]\label{thm:Williams QC theorem}
Let $1<Q<\infty$; let $Z$ and $W$ be separable, locally finite metric measure spaces;
and let $f\colon Z\to W$ be a homeomorphism. Then the following two conditions are equivalent,
with the same constant $K$:
\begin{enumerate}
	\item $f\in N_{\loc}^{1,Q}(Z;W)$ and for $\mu$-a.e. $x\in Z$,
	\[
	g_f(x)^Q\le K J_f(x);
	\]
	\item For every family $\Gamma$ of curves in $Z$,
	\[
	\Mod_Q(\Gamma)\le K\Mod_Q(f(\Gamma)).
	\]
\end{enumerate}
\end{theorem}

Note that $f(\Gamma)$ means the curves $f\circ\gamma$,
$\gamma\in\Gamma$, reparametrized by arc-length.
We will apply the implication $(1)\Rightarrow (2)$ with the choices $Z=\Om\subset X$
and $W=f(\Om)\subset Y$.
In \cite{Wi12} it is additionally assumed that the supports of the measures that
$Z$ and $W$ are equipped with are the entire spaces, but this is not needed in the proof of
$(1)\Rightarrow (2)$.

We define the Hardy--Littlewood maximal function of a locally integrable nonnegative function
$g\in L^1_{\loc}(X)$ by
\[
Mg(x):=\sup_{0<r<\infty}\,\vint{B(x,r)}g\,d\mu,\quad x\in X.
\]

The following fact is well known, but we present it in a slightly different form than what is usual,
so we also sketch a proof.
Recall the Poincar\'e inequality from \eqref{eq:Poincare}.

\begin{proposition}\label{prop:p to q}
	Let $1\le p< q<\infty$.
	Suppose $X$ is proper
	and supports a
	$(1,p)$-Poincar\'e inequality, and $\mu$ is doubling.
Let $\Om\subset X$ be open and suppose $f\colon \Om\to Y$ is continuous and that
$g\in L^q(\Om)$ is a $p$-weak upper gradient of $f$ in $\Om$. Then
$f\in D^q(\Om;Y)$.
\end{proposition}
\begin{proof}
We can interpret $g$ to be zero extended to the whole space, so that $g\in L^q(X)$.
Consider a ball $B(z,r)\subset B(z,3\lambda r)\subset \Om$,
where $\lambda\ge 1$ is the dilation constant from the Poincar\'e inequality.
By a telescoping argument, see e.g. \cite[Theorem 8.1.7]{HKSTbook}, for all $x,y\in B(z,r)$
we get
\[
d_Y(f(x),f(y))\le C d(x,y) ([Mg^p(x)]^{1/p}+[Mg^p(y)]^{1/p})
\]
for some $C>0$ that only depends on the constants of the doubling and Poincar\'e conditions.
Here $[Mg^p]^{1/p} \in L^q(X)$
by the Hardy--Littlewood maximal theorem, see e.g. \cite[Theorem 3.5.6]{HKSTbook}.
Moreover, $f\in L^q(B(x,r))$ since $X$ is proper and so $f$ is bounded in $B(x,r)$.
It follows that $f$ is in the \emph{Haj{\l}asz--Sobolev space} $M^{1,q}(B(z,r);Y)$.
Then by the proof of \cite[Lemma 10.2.5]{HKSTbook}, we know that $3C[Mg^p]^{1/p}$
is an upper gradient of $f$ in $B(z,r)$. Since we can cover $\Om$ by countably many
such balls, it follows that $3C[Mg^p]^{1/p}$
is an upper gradient of $f$ in $\Om$,
and so $f\in D^q(\Om;Y)$.
\end{proof}

\section{Proofs of the main results}\label{sec:proofs}

In this section we prove our main result, Theorem \ref{thm:main theorem intro},
and also Corollaries \ref{cor:weighted}, \ref{cor:Euclidean}, and \ref{cor:Carnot}.

Throughout this section we assume that
$\Om\subset X$ is nonempty, open, and bounded,
and that $f\colon \Om\to f(\Om)\subset Y$ is a homeomorphism with
$f(\Om)$ open and $\nu(f(\Om))<\infty$.

We will consider the following Lusin property on curves.

\begin{definition}
	Let $\gamma\colon [0,\ell_{\gamma}]\to X$ be a curve and let $A\subset X$, and $f\colon X\to Y$.
	We say that $f$ satisfies the 
	$N_A$-property on $\gamma$
	if for every $N\subset [0,\ell_{\gamma}]$ with $\mathcal L^1(N)=0$, we have
	\[
	\mathcal H^1(f(\gamma(N)\cap A))=0.
	\]
	We call the $N_X$-property simply the $N$-property.
\end{definition}

We give the proof of Theorem \ref{thm:main theorem intro} in the following two propositions.
The idea of separating the argument into two steps,
first considering absolute continuity on curves
and then the Dirichlet seminorm,
comes from Williams \cite{Wi}.

\begin{proposition}\label{prop:abs cont}
Suppose there exists a Borel regular outer measure $\widetilde{\mu}\ge \mu$ on $X$ which is doubling
within a ball $2B_0$ with $\Om\subset B_0$,
and that there exist a set $E\subset\Om$ and a function $Q(x)>1$
on $\Om\setminus E$, with
\[
\limsup_{r\to 0}\frac{\widetilde{\mu}(B(x,r))}{r^{Q(x)}}<\infty
\quad\textrm{and}\quad \liminf_{r\to 0}\frac{\nu(B(f(x),r))}{r^{Q(x)}}>0
\quad \textrm{for all }x\in \Om\setminus E.
\]
Suppose $Q:=\inf_{x\in\Om\setminus E}Q(x)>1$ and let $1\le p\le Q$.
Suppose also that
\begin{equation}\label{eq:E assumption}
\Mod_p(\{\gamma\subset\Om\colon \mathcal H^1(f(\gamma\cap E))>0\})=0
\end{equation}
and $h_f<\infty$ in $\Om\setminus E$.
Then $f$ is satisfies the $N$-property on $p$-a.e. curve $\gamma$ in $\Om$.
\end{proposition}

\begin{proof}
	
Define the sets $A_k$, $k\in \N$, as subsets of $\Om\setminus E$ such that $\lip_f(x)> 1$,
$Q(x)\le k$,
$h_f(x)\le k$, as well as

\begin{equation}\label{eq:k and 1 over k}
\limsup_{r\to 0}\frac{\widetilde{\mu}(B(x,r))}{r^{Q(x)}}< k
\quad\textrm{and}\quad
\liminf_{r\to 0}\frac{\nu(B(f(x),r))}{r^{Q(x)}}> 1/k\quad \textrm{for all }x\in A_k.
\end{equation}
Then $\Om\cap \{\lip_f>1\}\setminus E=\bigcup_{k=1}^{\infty}A_k$.

Fix $k\in \N$ and let $A:=A_k$.
Fix $\delta>0$.
For every $x\in A$ we can choose a radius $0<r_x<\delta$ sufficiently small so that
\begin{itemize}
	\item $B(x,2r_x)\subset \Om$,
	$B(f(x),L_f(x,r_x))\subset f(\Om)$, and $L_f(x,r_x)\le \delta/2$ (since $f$ is continuous);
	\item by the fact that $\lip_f(x)> 1$, we can get
	\begin{equation}\label{eq:Lf ge 1}
	\frac{L_f(x,r_x)}{r_x}> 1;
	\end{equation}
	\item by \eqref{eq:k and 1 over k},
	\begin{equation}\label{eq:sup and inf}
		\sup_{0<r\le r_x}\frac{\widetilde{\mu}(B(x,r))}{r^{Q(x)}}\le k
		\quad \textrm{and}\quad \inf_{0<r\le L_f(x,r_x)}\frac{\nu(B(f(x),r))}{r^{Q(x)}}\ge 1/k.
	\end{equation}
\end{itemize}
Finally since $h_f(x)=\liminf_{r\to 0} H_f(x,r)$,
and noting that $h_f(x)\ge 1$ for every $x\in\Om$ by the fact that $X$ is connected,
we can also choose $r_x$ to have
\begin{equation}\label{eq:Hf and hf}
	H_f(x,r_x)\le 2h_f(x)\le 2k.
\end{equation}
From the fact that $\widetilde{\mu}$ is doubling within some ball $2B_0$ with $A\subset \Om\subset B_0$,
we obtain that $(A,d)$ is metric doubling, see \cite[Proposition 3.4]{BBlocal}.
Thus we can apply Lemma \ref{lem:N countable collections}
to the covering $\mathcal G:=\{B(x,r_x)\}_{x\in A}$,
to extract subcoverings
$\mathcal G_{1},\ldots,\mathcal G_{N}$, with
\[
\mathcal G_{j}=\{B_{j,l}=B(x_{j,l},r_{j,l})\}_{l}
\]
and having the good properties given in the Lemma.
Define
\[
g:=2\sum_{j=1}^N\sum_{l}\frac{L_f(x_{j,l},r_{j,l})}{r_{j,l}}\ch_{2B_{j,l}}.
\]
Consider a curve $\gamma$ in $\Om$ with $\diam\gamma >\delta$.
If $\gamma$ intersects $B_{j,l}$, then $\mathcal H^1(\gamma\cap 2B_{j,l})>r_{j,l}$.
Thus we have
\begin{equation}\label{eq:upper gradient property proved}
	\int_{\gamma}g\,ds\ge 2\sum_{\gamma\cap B_{j,l}\neq\emptyset}L_f(x_{j,l},r_{j,l})
	\ge  \sum_{\gamma\cap B_{j,l}\neq\emptyset}\diam (f(B_{j,l}))
	\ge \mathcal H_{\delta}^{1}(f(\gamma\cap A)),
\end{equation}
where the last inequality holds since the balls $B_{j,l}$ satisfying $\gamma\cap B_{j,l}\neq\emptyset$
cover $\gamma\cap A$ and so the sets $f(B_{j,l})$ with $\gamma\cap B_{j,l}\neq\emptyset$ cover $f(\gamma\cap A)$.

Note that for each ball $B_{j,l}$, we have
\[
B\left(f(x_{j,l}),\frac{L_f(x_{j,l},r_{j,l})}{H_f(x_{j,l},r_{j,l})}\right)
= B(f(x_{j,l}),l_f(x_{j,l},r_{j,l}))\\
\subset f(B_{j,l}).
\]
Now for every $j\in\{1,\ldots,N\}$, Lemma \ref{lem:disjoint images} gives
\[
B\left(f(x_{j,l}),\frac{L_f(x_{j,l},r_{j,l})}{2H_f(x_{j,l},r_{j,l})}\right)\cap B\left(f(x_{j,m}),
\frac{L_f(x_{j,m},r_{j,m})}{2H_f(x_{j,m},r_{j,m})}\right)
=\emptyset
\quad\textrm{for all }l\neq m,
\]
and so by \eqref{eq:Hf and hf},
\begin{equation}\label{eq:disjoint images}
B\left(f(x_{j,l}),\frac{L_f(x_{j,l},r_{j,l})}{4k}\right)\cap B\left(f(x_{j,m}),
\frac{L_f(x_{j,m},r_{j,m})}{4k}\right)
=\emptyset
\quad\textrm{for all }l\neq m.
\end{equation}
Denote $Q_{j,l}:=Q(x_{j,l})$ and $L_{j,l}:=L(x_{j,l},r_{j,l})$.
For every $k\in\N$, abbreviating $\sum_{j=1}^N \sum_{l}$ by $\sum_{j,l}$,
from the fact that $\widetilde{\mu}\ge \mu$ we get
\begin{align*}
	\int_\Om g^Q\,d\mu
	\le \int_\Om g^Q\,d\widetilde{\mu}
	&= 2^Q \int_\Om \left(\sum_{j,l}\frac{L_{j,l}}{r_{j,l}}\ch_{2B_{j,l}}\right)^Q\,d\widetilde{\mu}\\
	&\le 2^Q C_0\int_\Om \left(\sum_{j,l}\frac{L_{j,l}}{r_{j,l}}\ch_{\tfrac 13 B_{j,l}}\right)^Q\,d\widetilde{\mu}
	\quad\textrm{by Lemma }\ref{lem:dilated balls lemma}\\
	&= 2^Q C_0 \sum_{j,l}\left(\frac{L_{j,l}}{r_{j,l}}\right)^Q\widetilde{\mu}(\tfrac 13 B_{j,l})\\
	&\le 2^Q C_0 \sum_{j,l}\left(\frac{L_{j,l}}{r_{j,l}}\right)^{Q_{j,l}}\widetilde{\mu}(\tfrac 13 B_{j,l})
\end{align*}
by \eqref{eq:Lf ge 1} and the fact that $Q\le Q_{j,l}$.
Using $\eqref{eq:sup and inf}$, we continue the estimate
\begin{align*}
	\int_\Om g^Q\,d\mu
	&\le 2^Q kC_0 \sum_{j,l}\left(\frac{L_{j,l}}{r_{j,l}}\right)^{Q_{j,l}} r_{j,l}^{Q_{j,l}}\\
	&= 2^Q kC_0 \sum_{j,l}L_{j,l}^{Q_{j,l}}\\
	&\le 2^Q kC_0 (4k)^k\sum_{j,l}\left(\frac{L_{j,l}}{4k}\right)^{Q_{j,l}}
	\quad \textrm{since }Q_{j,l}\le k\\
	&\le 2^Q k^2 (4k)^k C_0 \sum_{j,l}\nu(B(f(x_{j,l}),L_{j,l}/4k))\quad \textrm{by }
	\eqref{eq:sup and inf}\\
	&\le 2^Q k^2 (4k)^k C_0 \nu(f(\Om))\quad \textrm{by }\eqref{eq:disjoint images}\\
	&<\infty.
\end{align*}

Recall that $k\in\N$ is kept fixed, but $g$ depends on $\delta>0$.
Now we can choose functions $g$ with the choices $\delta=1/i$,
to get a sequence $\{g_i\}_{i=1}^{\infty}$ that is bounded in $L^{Q}(\Om)$.
By the reflexivity of the space $L^Q(\Om)$, we find a subsequence (not relabeled) and $g\in L^Q(\Om)$
such that $g_i\to g$ weakly in $L^Q(\Om)$.
By Mazur's and Fuglede's lemmas
(Lemma \ref{lem:Mazur lemma} and Lemma \ref{lem:Fuglede lemma}),
we find convex combinations $\widehat{g}_i:=\sum_{j=i}^{N_i}a_{i,j}g_j$ such that for
$Q$-a.e. curve $\gamma'$ in $\Om$ we have
\begin{equation}\label{eq:gamma prime ineq}
\begin{split}
\int_{\gamma'}g\,ds
=\lim_{i\to\infty}\int_{\gamma'}\widehat{g}_i\,ds
&\ge \lim_{i\to\infty}\mathcal H_{1/i}^{1}(f(\gamma'\cap A))
\quad \textrm{by }\eqref{eq:upper gradient property proved}\\
&= \mathcal H^{1}(f(\gamma'\cap A)).
\end{split}
\end{equation}
By the properties of modulus,
see e.g. \cite[Lemma 1.34]{BB},
for $Q$-a.e. curve $\gamma$ in $\Om$ we have that the above holds
for every subcurve $\gamma'$ of $\gamma$.
For $Q$-a.e. curve $\gamma$ in $\Om$, we also have that
$\int_{\gamma}g\,ds<\infty$, which follows from the definition of the $Q$-modulus.
Fix a curve $\gamma$ satisfying the above two conditions.
We can write any open $U\subset (0,\ell_{\gamma})$ as a union of pairwise disjoint intervals
$U=\bigcup_{j=1}^{\infty}(a_j,b_j)$, and then
\begin{equation}\label{eq:gamma A ineq}
\begin{split}
\int_{U}g(\gamma(s))\,ds
=\sum_{j=1}^{\infty}\int_{(a_j,b_j)}g(\gamma(s))\,ds
&\ge \sum_{j=1}^{\infty}\mathcal H^{1}(f(\gamma((a_j,b_j))\cap A))\quad\textrm{by }\eqref{eq:gamma prime ineq}\\
&\ge \mathcal H^{1}(f(\gamma(U)\cap A))
\end{split}
\end{equation}
by the subadditivity of the $\mathcal H^1$-measure.
Then for any Borel set $S\subset (0,\ell_{\gamma})$, we can let $\eps>0$ and find an open set $U$ such that
$S\subset U\subset (0,\ell_{\gamma})$ and
\begin{align*}
\int_{S}g(\gamma(s))\,ds
\ge \int_{U}g(\gamma(s))\,ds-\eps
&\ge \mathcal H^{1}(f(\gamma(U)\cap A))-\eps\quad\textrm{by }\eqref{eq:gamma A ineq}\\
&\ge \mathcal H^{1}(f(\gamma(S)\cap A))-\eps.
\end{align*}
Letting $\eps\to 0$, we get
\[
\int_{S}g(\gamma(s))\,ds
\ge \mathcal H^{1}(f(\gamma(S)\cap A)),
\]
which in particular proves the $N_A$-property for $Q$-a.e. curve $\gamma$ in $\Om$.
Then the property also holds for $p$-a.e. curve $\gamma$ in $\Om$, since $p\le Q$, see e.g.
\cite[Proposition 2.45]{BB}.
Recall that so far $k\in\N$ was fixed and $A=A_k$.\\

In total, since $\Om=\bigcup_{k=1}^{\infty}A_k\cup E\cup \{\lip_f\le 1\}$, for $p$-a.e. curve $\gamma$
in $\Om$ we have that if $N\subset [0,\ell_{\gamma}]$ with $\mathcal L^1(N)=0$, then
\begin{align*}
&\mathcal H^{1}(f(\gamma(N)))\\
&\quad \le \sum_{k=1}^{\infty}\mathcal H^{1}(f(\gamma(N)\cap A_k))+\mathcal H^{1}(f(\gamma(N)\cap\{\lip_f\le 1\}))
+\mathcal H^{1}(f(\gamma\cap E))\\
&\quad =0+0+0
\end{align*}
by the $N_{A_k}$-property proved just above, Lemma \ref{lem:abs cont},
and the assumption \eqref{eq:E assumption}.
Thus $f$ satisfies the $N$-property on $p$-a.e. curve $\gamma$ in $\Om$.
\end{proof}

\begin{proposition}\label{prop:lip}
	Suppose there exists a Borel regular outer measure
	$\widetilde{\mu}\ge \mu$ on $X$ which is doubling within $\Om$.
		Suppose also there exists a $\mu$-measurable set $E\subset\Om$ and $\mu$-measurable
		functions $Q(x)>1$ and $R(x)>0$ in $\Om\setminus E$ such that
		\begin{equation}\label{eq:first limsup and liminf}
		\limsup_{r\to 0}\frac{\mu(B(x,r))}{r^{Q(x)}}<R(x)\liminf_{r\to 0}\frac{\nu(B(f(x),r))}{r^{Q(x)}}
		\quad \textrm{for }\mu\textrm{-a.e. }x\in \Om\setminus E.
		\end{equation}
		Suppose $Q:=\inf_{x\in\Om\setminus E}Q(x)>1$ and let $1\le q\le Q$.
		Assume also that
		\[
		\begin{cases}
		\frac{Q(\cdot)-q}{Q(\cdot)}(R(\cdot)h_f(\cdot)^{Q(\cdot)})^{q/(Q(\cdot)-q)}\in L^1(
		\Om\setminus E)\quad\textrm{if }1\le q<Q;\\
		R(\cdot)^{1/Q(\cdot)}h_f(\cdot)\in L^{\infty}(\Om\setminus E)\quad\textrm{if }q=Q.
		\end{cases}
		\]
		Then $\lip_f \in L^q(\Om\setminus E)$.
		In the case $q=Q=Q(x)$ for $\mu$-a.e. $x\in \Om\setminus E$, we also get
		\[
		\lip_f(x)^Q\le \Vert R(\cdot)h_f(\cdot)^Q\Vert_{L^{\infty}(\Om)} J_f(x)\quad\textrm{for }
		\mu\textrm{-a.e. }x\in\Om\setminus E.
		\]
\end{proposition}	
\begin{proof}
For $\mu$-a.e. $x\in \Om\setminus E$,
by \eqref{eq:first limsup and liminf} we have for some $C(x)>0$ and some sufficiently small $r_x>0$ that
\begin{equation}\label{eq:rx and lfxrx condition}
\frac{\mu(B(x,r_x))}{r_x^{Q(x)}}< C(x) R(x)\quad\textrm{and}\quad
\frac{\nu(B(f(x),l_f(x,r_x)))}{l_f(x,r_x)^{Q(x)}}> C(x).
\end{equation}
Fix $\eps>0$.
Since $h_f(x)=\liminf_{r\to 0} H_f(x,r)$,
and noting that $h_f(x)\ge 1$ for every $x\in\Om$,
we can choose $r_x$ so that we also have
\begin{equation}\label{eq:Hf and hf 2}
H_f(x,r_x)\le (1+\eps)h_f(x).
\end{equation}
Note that $B(f(x),l_f(x,r_x))\subset f(B(x,r_x))$.
Thus we estimate
\begin{equation}\label{eq:Jacobian estimate}
\begin{split}
\frac{\nu(f(B(x,r_x)))}{\mu(B(x,r_x))}
&\ge \frac{\nu(B(f(x),l_f(x,r_x)))}{\mu(B(x,r_x))}\\
&\ge \frac{C(x) l_f(x,r_x)^{Q(x)}}{C(x) R(x) r_x^{Q(x)}}\quad\textrm{by }\eqref{eq:rx and lfxrx condition}\\
&= \frac{1}{R(x)}\left(\frac{l_f(x,r_x)}{r_x}\right)^{Q(x)}\\
&\ge \frac{1}{R(x)}\left(\frac{L_f(x,r_x)}{H_f(x,r_x)\cdot r_x}\right)^{Q(x)}\\
&\ge \frac{1}{R(x)(1+\eps)^{{Q(x)}}}\left(\frac{L_f(x,r_x)}{h_f(x)\cdot r_x}\right)^{Q(x)}
\end{split}
\end{equation}
by \eqref{eq:Hf and hf 2}.
Recall from \eqref{eq:integral of Jf} that the Jacobian $J_f$ exists $\mu$-a.e. in $\Om$.
Since \eqref{eq:Jacobian estimate} holds for arbitrarily small $r_x$, we can take
the limit $\liminf_{r_x\to 0}$ to obtain at $\mu$-a.e. $x\in\Om\setminus E$ that
\[
J_f(x)\ge \frac{1}{R(x)(1+\eps)^{{Q(x)}}}\frac{\lip_f(x)^{Q(x)}}{h_f(x)^{Q(x)}}.
\]
Thus
\begin{equation}\label{eq:lipf estimate}
\lip_f(x)\le (1+\eps) R(x)^{1/Q(x)}h_f(x)J_f(x)^{1/Q(x)}
\end{equation}
and so for any $1\le q< Q$, we get by Young's inequality
\begin{align*}
\lip_f(x)^q
&\le (1+\eps)^q R(x)^{q/Q(x)}h_f(x)^q J_f(x)^{q/Q(x)}\\
&\le \frac{Q(x)-q}{Q(x)}\big(R(x)(1+\eps)^{Q(x)}h_f(x)^{Q(x)}\big)^{q/(Q(x)-q)} + J_f(x),
\end{align*}
where we estimated simply $q/Q(x)\le 1$ for the second term.
Using also \eqref{eq:integral of Jf}, we conclude
\begin{align*}
&\int_{\Om\setminus E}\lip_f(x)^q\,d\mu(x)\\
&\qquad \le (1+\eps)^{qQ/(Q-q)}\int_{\Om\setminus E}\frac{Q(x)-q}{Q(x)}(R(x)h_f(x)^{Q(x)})^{q/(Q(x)-q)}\,d\mu(x)\\
&\qquad\qquad\qquad\qquad\qquad\qquad\qquad\qquad\qquad\qquad + \nu(f(\Om))<\infty
\end{align*}
by assumption.

In the case $q=Q$, from \eqref{eq:lipf estimate} we estimate simply
\begin{equation}\label{eq:lipf Jf inequality}
\lip_f(x)^Q\le \Vert (1+\eps) R(\cdot)^{1/Q(\cdot)}h_f(\cdot)\Vert_{L^{\infty}(\Om\setminus E)}^Q J_f(x)^{Q/Q(x)}
\quad\textrm{for }\mu\textrm{-a.e. }x\in \Om\setminus E.
\end{equation}
Since $Q/Q(x)\le 1$ and $J_f\in L^1(\Om)$, and $\mu(\Om)<\infty$, also
$J_f(\cdot)^{Q/Q(\cdot)}\in L^1(\Om)$.
Thus $\lip_f\in L^Q(\Om\setminus E)$.

In the case $Q(x)=Q$ for $\mu$-a.e. $x\in \Om$, \eqref{eq:lipf Jf inequality} gives
\[
\lip_f(x)^Q\le \Vert (1+\eps) R(\cdot)h_f(\cdot)^Q\Vert_{L^{\infty}(\Om)}
 J_f(x)\quad\textrm{for }
\mu\textrm{-a.e. }x\in\Om\setminus E.
\]
Letting $\eps\to 0$, we get the last conclusion.
\end{proof}

\begin{proof}[Proof of Theorem \ref{thm:main theorem intro}]
By Proposition \ref{prop:lip} we know that $\lip_f(x)<\infty$ for $\mu$-a.e. $x\in\Om\setminus E$,
and so we know that for $p$-a.e. curve $\gamma$ in $\Om$, we have $\mathcal L^1(N_{\gamma})=0$ for
\[
N_{\gamma}:=\{t\in [0,\ell_{\gamma}]\colon \gamma(t)\in \Om\setminus E\textrm{ and }\lip_f(\gamma(t))=\infty\}.
\]
By Proposition \ref{prop:exceptional set E} combined with
Proposition \ref{prop:abs cont}, for $p$-a.e. curve $\gamma$ in $\Om$
we thus have $\mathcal H^1(f(\gamma(N_{\gamma})))=0$.
Denoting the end points of $\gamma$ by $x,y$, we get
\begin{align*}
d_Y(f(x),f(y))
&\le \mathcal H^1(f(\gamma([0,\ell_{\gamma}]\setminus N_{\gamma}))\\
&= \mathcal H^1(f(\gamma([0,\ell_{\gamma}]\setminus N_{\gamma})\setminus E))\quad\textrm{by Proposition \ref{prop:exceptional set E}}\\
&\le \int_{\gamma}\lip_f \ch_{\Om\setminus E}\,ds
\end{align*}
by Lemma \ref{lem:abs cont} with $A=[0,\ell_{\gamma}]\cap \gamma^{-1}(\Om\setminus E)\setminus N_{\gamma}$.
Thus $\lip_f \ch_{\Om\setminus E}$ is a $p$-weak upper gradient of $f$ in $\Om$.
We also have $\lip_f \ch_{\Om\setminus E}\in L^q(\Om)\subset L^p(\Om)$ by Proposition \ref{prop:lip},
so we conclude that $f\in D^p(\Om;Y)$.

Since $\lip_f \ch_{\Om\setminus E}$ is a $p$-weak upper gradient of $f$ in $\Om$,
for the minimal $p$-weak upper gradient we get
$g_f\le \lip_f \ch_{\Om\setminus E}$ $\mu$-a.e. in $\Om$.
In the case $p=Q=Q(x)$ for $\mu$-a.e. $x\in \Om\setminus E$, Proposition \ref{prop:lip} now gives
for $\mu$-a.e. $x\in \Om$ that
\[
g_f(x)^Q\le \Vert R(\cdot)h_f(\cdot)^Q\Vert_{L^{\infty}(\Om)} J_f(x).
\]
Then Theorem \ref{thm:Williams QC theorem}
gives for every curve family $\Gamma$ in $\Om$ that
\[
\Mod_Q(\Gamma)\le \Vert R(\cdot)h_f(\cdot)^{Q}\Vert_{L^{\infty}(\Om)}\Mod_Q(f(\Gamma)).
\]

Finally, if $X$ is proper and supports a
$(1,p)$-Poincar\'e inequality, and $\mu$ is doubling, then
by Proposition \ref{prop:p to q} we get $f\in D^{q}(\Om;Y)$.
\end{proof}

Next we consider weighted spaces.
By a weight we simply mean a nonnegative locally integrable function.
Suppose $Y$ is equipped with an Ahlfors regular measure $\nu_0$, and then we add a weight $w_Y$.
Note that $w_Y$ has Lebesgue points $\nu_0$-a.e., see e.g. Heinonen \cite[Theorem 1.8]{Hei}.
For our purposes it is natural to consider the pointwise representative
\begin{equation}\label{eq:wY representative}
w_Y(y)=\liminf_{r\to 0}\frac{1}{\nu_0(B(y,r))}\int_{B(y,r)}w_Y\,d\nu_0,\quad y\in Y.
\end{equation}

\begin{proof}[Proof of Corollary \ref{cor:weighted}]
We can choose $R(x):=(1+\eps)w(x)/w_Y(f(x))$ for arbitrarily small $\eps>0$,
and then the result follows from Theorem \ref{thm:main theorem intro}.
\end{proof}

\begin{proof}[Proof of Corollary \ref{cor:Euclidean} and Corollary \ref{cor:Carnot}]
The fact that $f\in D^q(\Om;\R^n)$
(resp. $f\in D^q(\Om;G)$) follows from Corollary \ref{cor:weighted}, since the Euclidean space
(resp. Carnot group), equipped with
the Lebesgue measure (resp. $Q$-dimensional Hausdorff measure),
support a $(1,1)$-Poincar\'e inequality and are equipped with
an Ahlfors $n$-regular  (resp. $Q$-regular), and hence doubling, measure.

In the case $q=n$ (resp. $q=Q$),
by Proposition \ref{prop:lip} we get
\[
\lip_f(x)^{q} \le \Vert w_Y(f(\cdot))^{-1}h_f(\cdot)^q \Vert_{L^{\infty}(\Om)}J_f(x)<\infty
\]
for $\mathcal L^n$-a.e. (resp. $\mathcal H^Q$-a.e.) $x\in \Om$.
Thus for $q$-a.e. curve $\gamma$ in $\Om$, we have
\[
\mathcal L^1(\{t\in [0,\ell_{\gamma}]\colon \lip_f(\gamma(t))=\infty\})=0.
\]
Since $f\in D^{q}(\Om;Y)$, $f\circ\gamma\colon [0,\ell_{\gamma}]\to Y$ is absolutely continuous
for $q$-a.e. curve $\gamma$ in $\Om$, and then by Lemma \ref{lem:abs cont}
we find that $\lip_f$ is a $q$-weak upper gradient of $f$ in $\Om$.
Thus
\[
g_f(x)^{q} \le \Vert w_Y(f(\cdot))^{-1}h_f(\cdot)^q \Vert_{L^{\infty}(\Om)}J_f(x)
\]
for $\mathcal L^n$-a.e. (resp. $\mathcal H^Q$-a.e.) $x\in \Om$,
and from Theorem \ref{thm:Williams QC theorem} we get
\[
\Mod_q(\Gamma)\le \Vert w_Y(f(\cdot))^{-1} h_f(\cdot)^{q}\Vert_{L^{\infty}(\Om)}\Mod_q(f(\Gamma))
\]
for every curve family $\Gamma$ in $\Om$.
\end{proof}

\begin{remark}
Concerning the assumption $\nu(f(\Om))<\infty$ that we make throughout, note that 
as a continuous mapping $f$ is bounded in every $\Om'\Subset\Om$
(i.e. $\overline{\Om'}$ is a compact subset of $\Om$).
Thus we have $\nu(f(\Om'))<\infty$ and also $f\in L^p(\Om',Y)$,
so if $f\in D^p(\Om';Y)$ then in fact $f\in N^{1,p}(\Om';Y)$.
But since we do not assume $X$ to be proper, there may not be many such sets $\Om'$,
and so we prefer to simply assume $\nu(f(\Om))<\infty$.
\end{remark}

\section{Examples}\label{sec:examples}

In this section we discuss various examples and applications of our main results.

Corollary \ref{cor:weighted} applies to a wide range of weighted spaces.
For example, in $\R^n$ we can consider any weight $w\le \widetilde{w}$, where $\widetilde{w}$
is $p$-admissible or a $p$-Muckenhoupt weight for which
\[
\limsup_{r\to 0}\frac{1}{\mathcal L^n(B(x,r))}\int_{B(x,r)}\widetilde{w}\,d\mathcal L^n<\infty
\]
apart from at most countably many $x\in\R^n$.
See e.g. \cite[Appendix A.2]{BB} for a discussion on these concepts.

We recall that part of the analytic definition of quasiconformality is that $f\in N_{\loc}^{1,Q}(X;Y)$,
and so the case $p=Q$ is of particular interest in the theory.
To the best of our knowledge, in all previous results the assumption
$h_f\in L^{\infty}(\Om)$ has been made in order to obtain $f\in N_{\loc}^{1,Q}(X;Y)$ or $f\in D_{\loc}^{Q}(X;Y)$.
As an elementary application of our results, we note that this strong assumption
can be relaxed if the weight is small where $h_f$ is large.
The following corollary is essentially the case $p=Q$ of
\cite[Corollary 1.3]{Wi}, but there the weight was simply $w=1$.

\begin{corollary}\label{cor:basic}
Let $Q>1$.
Let $(X_0,d,\mu_0)$ and $(Y,d_Y,\nu)$ be Ahlfors $Q$-regular spaces.
Let $X=X_0$ as a metric space
but equipped with the weighted measure $d\mu=w\,d\mu_0$, with $0\le w\le 1$.
Let $\Om\subset X$ be open and bounded and let
$f\colon \Om\to f(\Om)\subset Y$ be a homeomorphism with $f(\Om)$ open and $\nu(f(\Om))<\infty$.
Suppose there is an at most countable set $E\subset \Om$ such that
$h_f<\infty$ in $\Om\setminus E$, and
$w(\cdot)h_f(\cdot)^{Q}\in L^{\infty}(\Om)$.
Then $f\in D^{Q}(\Om;Y)$.
\end{corollary}
\begin{proof}
	This follows from Corollary \ref{cor:weighted};
note that we can choose $\widetilde{w}\equiv 1$.
\end{proof}

Next we give a more concrete example.

\begin{example}\label{ex:concrete}
Consider the  square $\Om:=(-1,1)\times (-1,1)$ on the unweighted plane
$X=Y=\R^2$.
Let $0<b\le 1/2$ and consider the homeomorphism $f\colon \Om\to \Om$
\[
f(x_1,x_2):=
\begin{cases}
(x_1,x_2^{b}),\quad x_2\ge 0,\\
(x_1,-|x_2|^{b}),\quad x_2\le 0.
\end{cases}
\]
Essentially, $f$ maps squares centered at the origin to rectangles that become
thinner and thinner near the origin.

By symmetry, it will be enough to study the behavior of $f$ in the unit square $S:=(0,1)\times (0,1)$.
There we have
\[
Df(x_1,x_2)=
\left[
\begin{matrix}
1 & 0 \\
0 & b x_2^{b-1}
\end{matrix}
\right],
\]
and so
\[
|Df|=\sqrt{(1+(b x_2^{b-1})^2}\ge b x_2^{b-1}.
\]
Thus $|Df|\notin L^2(S)$ so that $f$ is not in the classical Dirichlet space $D^2_{\textrm{euc}}(S;S)$,
and then $f\notin D^{2}(S;S)$ by \cite[Corollary A.4]{BB}.
In particular, $f\notin D^{2}(\Om;\Om)$.

Clearly $f$ maps a small square centered at $(x_1,x_2)\in S$ with side length $\eps$
to a rectangle centered at $f(x_1,x_2)$ and with side lengths
$\eps$ and $b x_2^{b-1}\eps+o(\eps)$.
This means that
\[
h_f(x_1,x_2)= bx_2^{b-1}\quad\textrm{in }S_1:=
\{(x_1,x_2) \in S\colon x_2\le b^{1/(1-b)}\}
\]
and
\[
h_f(x_1,x_2)= b^{-1}x_2^{1-b}\quad\textrm{in }S_2:=
\{(x_1,x_2) \in S\colon x_2\ge b^{1/(1-b)}\}.
\]
Of these two sets, $S_1$ is the relevant one for us, since it contains a neighborhood of
the $x_1$-axis in $S$.
Indeed, $h_f$ blows up on the $x_1$-axis and so it is not in $L^{\infty}(\Om)$.
Thus the conditions of Corollary \ref{cor:weighted} with $p=q=2$ are not fulfilled,
as of course they should not be.

On the other hand, let $X=\R^2$ equipped with a weighted Lebesgue measure,
with weight $w(x_1,x_2)=\min\{1,|x_2|^{2-2b}\}$.
Now for $(x_1,x_2)\in S_1$, we have
\[
w(x_1,x_2)h_f(x_1,x_2)^{2}=b^2,
\]
and then it is clear that $wh_f^2\in L^{\infty}(S)$, and then by symmetry $wh_f^2\in L^{\infty}(\Om)$.
Moreover, $h_f<\infty$ in $\Om\setminus E$, with $E$ consisting of the $x_1$-axis
intersected with $\Om$.
We let $\widetilde{w}:=w$ and note that the corresponding weighted Lebesgue measure is doubling.
Then is straightforward to check that
$E$ has finite $\widetilde{\mathcal H}^2$-measure
(the codimension $2$ Hausdorff measure with respect to $\widetilde{w}\,d\mathcal L^2$).
Now Corollary \ref{cor:weighted} gives $f\in D^2(\Om;\Om)$ and then in fact $f\in N^{1,2}(\Om;\Om)$.

We conclude that in a suitable weighted space, we can show $f$ to be a Newton-Sobolev mapping despite the fact
that $h_f$ is not essentially bounded.
\end{example}

In the setting of Theorem \ref{thm:main theorem intro}, we can also consider spaces whose
dimension varies between different parts of the space.
In \cite[Example 6.2]{LaZh} we consider some such spaces in the case $p=1$;
in the case $1<p\le Q$ the situation is quite similar so we do not go into more detail here.

Whereas Corollary \ref{cor:basic} discusses equipping $X$ with a (small) weight,
it is perhaps even more interesting to equip the space $Y$ with a weight.
This can be done in a very flexible way, since the space $D^p(\Om;Y)$ does not depend
on the measure $\nu$ that $Y$ is equipped with.
In this way, we can often avoid the strong assumption $h_f\in L^{\infty}(\Om)$
even in unweighted Euclidean spaces.
Below we use almost the same construction as in Example \ref{ex:concrete}; this is also
similar to \cite[Example 6.8]{LaZh} where we considered the case $p=1$.

\begin{example}\label{ex:plane}
	Consider the  square $\Om:=(-1,1)\times (-1,1)$ on the unweighted plane
	$X=Y=\R^2$.
	Now choose $1<b<\infty$ and consider the homeomorphism $f\colon \Om\to \Om$
	\[
	f(x_1,x_2):=
	\begin{cases}
	(x_1,x_2^{b}),\quad x_2\ge 0,\\
	(x_1,-|x_2|^{b}),\quad x_2\le 0.
	\end{cases}
	\]
	In the unit square $S:=(0,1)\times (0,1)$, we have
	\[
	Df(x_1,x_2)=
	\left[
	\begin{matrix}
	1 & 0 \\
	0 & b x_2^{b-1}
	\end{matrix}
	\right],
	\]
	and so
	\[
	|Df|=\sqrt{(1+(b x_2^{b-1})^2}\le 1+b x_2^{b-1}.
	\]
	Thus $|Df|\in L^2(S)$ and then in fact $|Df|\in L^2(\Om)$ and $f\in N^{1,2}(\Om;\Om)$.
	
	Moreover,
	\[
	h_f(x_1,x_2)= b^{-1}x_2^{1-b}\quad\textrm{in }S_1:=
	\{(x_1,x_2) \in S\colon x_2\le b^{1/(1-b)}\}
	\]
	and
	\[
	h_f(x_1,x_2)= bx_2^{b-1}\quad\textrm{in }S_2:=
	\{(x_1,x_2) \in S\colon x_2\ge b^{1/(1-b)}\}.
	\]
	Again, $S_1$ is the relevant set for us.
	Obviously $h_f\notin L^{\infty}(S_1)$ and then $h_f\notin L^{\infty}(\Om)$,
	since $h_f$ blows up on the $x_1$-axis.
	Thus the previous results in the literature
	do not detect that $f\in N^{1,2}(\Om;\Om)$.
	
	On the other hand, we can equip $Y$ with the weight $w_Y(y_1,y_2)=|y_2|^{u}$, $-1<u<0$.
	Now for $(x_1,x_2)\in \Om$,
	\[
	w_Y(f(x))=||x_2|^{b}|^{u}=|x_2|^{bu},
	\]
	and so for $(x_1,x_2)\in S_1$,
	\begin{equation}\label{eq:weight quantity}
	\frac{h_f(x_1,x_2)^{2}}{w_Y(f(x_1,x_2))}
	=b^{-2}x_2^{2-2b}x_2^{-bu}
	=b^{-2} x_2^{2-2b-bu},
	\end{equation}
	which is constant and thus in $L^{\infty}(S_1)$ if  $b= 2/(2+u)$.
	Going over the values $-1<u<0$, we conclude that all values $1<b<2$ are allowed.
	Since $S_1$ contains a neighborhood of the $x_1$-axis in $S$, clearly
	the quantity $h_f(x_1,x_2)^{2}/w_Y(f(x_1,x_2))$ is then in $L^{\infty}(S)$, and then by symmetry
	in $L^{\infty}(\Om)$.
	Moreover, $h_f<\infty$ in $\Om\setminus E$, with $E$ consisting of the $x_1$-axis
	intersected with $\Om$, so that
	$E$ has finite $\mathcal H^1$-measure.
	Thus Corollary \ref{cor:Euclidean} implies that $f\in N^{1,2}(\Om;\Om)$,
	and that for every curve family $\Gamma$ in $\Om$ we have
	\begin{equation}\label{eq:one sided conclusion}
	\Mod_2(\Gamma)\le C\Mod_2(f(\Gamma))
	\end{equation}
	with $C=\Vert w_Y(f(\cdot))^{-1} h_f(\cdot)^{2}\Vert_{L^{\infty}(\Om)}$.
	
	Consider the curve families
	\[
	\Gamma_i:=\{\gamma_t(s)=(t,s),\ 0<s<1/i\}_{t\in (-1,1)},\quad i\in\N.
	\]
	The families $f(\Gamma_i)$, when reparametrized by arc-length, are
	\[
	f(\Gamma_i)=\{\gamma_t(s)=(t,s),\ 0<s<1/i^b\}_{t\in (-1,1)}.
	\]
	The function $\rho_i=i \ch_{(-1,1)\times (0,1/i)}$ is admissible for $\Gamma_i$, and so
	\[
	\Mod_2(\Gamma_i)\le \int_{\Om}\rho_i^2\,d\mathcal L^2=2i.
	\]
	Then suppose $g_i$ is an admissible function for $\Mod_2(f(\Gamma_i))$.
	Here we again consider $Y=\R^2$ equipped with the Lebesgue measure $\mathcal L^2$.
	Using H\"older's inequality and then Fubini's theorem we estimate
	\begin{align*}
	\frac{2}{i^b}\int_{(-1,1)\times (0,1/i^b)}g_i^2\,d\mathcal L^2
	&\ge \left(\int_{(-1,1)\times (0,1/i^b)}g_i\,d\mathcal L^2\right)^2\\
	&= \left(\int_{-1}^1 \int_{0}^{1/i^b} g_i(x_1,x_2)\,dx_2\,dx_1 \right)^2\\
	&\ge \left(\int_{-1}^1 \,dx_1 \right)^2=4.
	\end{align*}
	In total, we get
	\[
	\Mod_2(\Gamma_i)\le 2i\quad\textrm{and}\quad \Mod_2(f(\Gamma_{i}))\ge 2i^b,
	\]
	where the latter is much larger when $i$ is large.
	Thus there can be no constant $C$
	for which we would have the reverse inequality to \eqref{eq:one sided conclusion},
	namely:
	\[
	\Mod_2(f(\Gamma))\le C\Mod_2(\Gamma)
	\]
	for every curve family $\Gamma$ in $\Om$.
	Of course, the same will happen if we equip $Y$ with the larger weighted measure 
	$w_Y\,d\mathcal L^2$ as we did above.
	This ``shortcoming'' is  natural: $h_f$ not being essentially bounded means that $f$ is certainly not
	a (metric) quasiconformal mapping.
	But we were still able to detect that $f\in N^{1,2}(\Om;\Om)$, as well as prove the 
	one-sided inequality \eqref{eq:one sided conclusion},
	solely by using information about the size of $h_f$.
\end{example}

\end{document}